\documentclass[11pt,reqno]{article}
\usepackage{latexsym, amsmath, amssymb, amsthm, a4, epsfig}
\usepackage{graphicx}
\newtheorem{theorem}{Theorem}[section]
\newtheorem{cor}[theorem]{Corollary}

\newtheorem{prop}[theorem]{Proposition}

\newtheorem{conjecture}{Conjecture}

\newcommand{\p}{\partial}

\newcommand{\eqnref}[1]{(\ref {#1})}

\newcommand{\Rbb}{\mathbb{R}}

\newcommand{\Kcal}{\mathcal{K}}

\newcommand{\Scal}{\mathcal{S}}

\def\Bn{{\bf n}}

\newcommand{\Ga}{\alpha}

\newcommand{\Ge}{\epsilon}

\newcommand{\Gl}{\lambda}

\newcommand{\Go}{\omega}

\newcommand{\GD}{\Delta}

\newcommand{\GL}{\Lambda}

\newcommand{\GO}{\Omega}

\newcommand{\beq}{\begin{equation}}
\newcommand{\eeq}{\end{equation}}

\numberwithin{equation}{section}
\numberwithin{figure}{section}

\begin{document}

\title{The first Hadamard variation of Neumann--Poincar\'e eigenvalues on the sphere\thanks{HK is supported by NRF 2016R1A2B4011304 and 2017R1A4A1014735}}

\author{Kazunori Ando\thanks{Department of Electrical and Electronic Engineering and Computer Science, Ehime University, Ehime 790-8577, Japan. Email: {\tt ando@cs.ehime-u.ac.jp}.}
\and Hyeonbae Kang\thanks{Department of Mathematics and Institute of Applied Mathematics, Inha University, Incheon 22212, S. Korea. Email: {\tt hbkang@inha.ac.kr}.}
\and Yoshihisa Miyanishi\thanks{Center for Mathematical Modeling and Data Science, Osaka University, Osaka 560-8531, Japan. Email: {\tt miyanishi@sigmath.es.osaka-u.ac.jp}.}
\and Erika Ushikoshi\thanks{Faculty of Environment and Information Sciences, Yokohama National University, Kanagawa 240-8501, Japan.
%79-7 Tokiwadai, Hodogaya-ku, Yokohama 240-8501, Japan.
Email: {\tt ushikoshi-erika-ng@ynu.ac.jp}.}
}

%%%---------------------------------------------------------------------------------------

\date{}
\maketitle

\begin{abstract}
The Neumann--Poincar\'e operator on the sphere has $\frac{1}{2(2k+1)}$, $k=0,1,2,\ldots$, as its eigenvalues and the corresponding multiplicity is $2k+1$.
We consider the bifurcation of eigenvalues under deformation of domains, and show that Frech\'et derivative of the sum of the bifurcations is zero. We then discuss
the connection of this result with some conjectures regarding the Neumann--Poincar\'e operator.
\end{abstract}

\noindent{\footnotesize {\bf AMS subject classifications}. 47A45 (primary), 31B25 (secondary)}

\noindent{\footnotesize {\bf Key words}. Neumann-Poincar\'e operator, eigenvalues, Hadamard's variation forumulas, Plasmon eigenvalues}

%\tableofcontents

%%%%%%%%%%%%%%%%%%%%%%%%%%%%%%%%%%%%%
\section{Introduction and statement of the result}
%%%%%%%%%%%%%%%%%%%%%%%%%%%%%%%%%%%%%

Let $\GO$ be a bounded domain in $\Rbb^3$ with the $C^{1, \Ga}$-smooth boundary for some $\Ga>0$. The Neumann-Poincar\'e (abbreviated by NP) operator  $\Kcal_{\p \GO}^*$ on $H^{-1/2}(\p \GO)$ or on $L^2(\p \GO)$ is defined by
\beq
\Kcal_{\p \GO}^*[\psi](x) := \int_{\p \GO}
\p_{\Bn_x} E(x, y) \psi(y) \; dS_y,
\eeq
where $\p_{\Bn_x}$ is the outer normal derivative (with respect to $x$-variables) on $\p \GO$, $E$ is the fundamental solution to the Laplacian in three dimensions, namely,
\beq
E(x, y)=\frac{-1}{4\pi} \frac{1}{|x-y|},
\eeq
and $dS_y$ is the surface element.

The NP operator appears naturally when solving classical boundary value problems in terms of layer potentials. For example, its relation to the single layer potential $\Scal_{\p \GO}$, which is defined by,
\beq
\Scal_{\p\GO}[\psi](x) := \int_{\p \GO} E(x, y) \psi(y) \; dS_y, \quad x \in \Rbb^3,
\eeq
is given by the following jump relation
\beq\label{singlejump}
\p_\Bn \Scal_{\p\GO}[\psi]|_{\pm} (x) = \left( \pm \frac{1}{2} I + \Kcal_{\p \GO}^* \right)[\psi](x), \quad x \in \p\GO,
\eeq
where the subscript $\pm$ on the left-hand side respectively denotes the limit (to $\p\GO$) from the outside and inside of $\GO$.

Even though $\Kcal_{\p \GO}^*$ is {\it not} self-adjoint on $L^2(\p \GO)$, namely, $\Kcal_{\p \GO}^* \neq \Kcal_{\p \GO}$, unless $\p\GO$ is a circle or a sphere, it is shown in \cite{KPS} that $\Kcal_{\p \GO}^*$ can be symmetrized on the Sobolev space $H^{-1/2}(\p\GO)$ by introducing a new but equivalent inner product on $H^{-1/2}(\p\GO)$. Since $\p\GO$ is $C^{1,\Ga}$, $\Kcal_{\p \GO}^*$ is compact on $H^{-1/2}(\p\GO)$. So $\Kcal_{\p \GO}$, as a compact self-adjoint operator on $H^{-1/2}(\p\GO)$, has eigenvalues converging to $0$. We refer such eigenvalues as NP eigenvalues.

It is proved by Poincar\'e \cite{Poincare-AM-87} that the NP eigenvalues on a ball are $\frac{1}{2(2k+1)}$ for $k=0, 1,2,\ldots$ and their multiplicities are $2k+1$ (see also the end of section \ref{sec:Hadamard variation}). So we may enumerate them as
\beq\label{eigensphere}
\Gl_{k,l}=\frac{1}{2(2k+1)}, \quad l=-k, \ldots, k.
\eeq
In this paper, we consider the variation of the NP eigenvalues on a ball. For a small number $h$ let $\GO(h)$ be a domain obtained by perturbing the boundary $\p\GO$ in the normal direction $\Bn$ by $ha$, namely,
\beq
\p\GO(h) = \{ x + ha(x) \Bn(x)\ :\ x \in \p\GO\}.
\eeq
If $\GO(0) = \GO$ is a ball, then the NP eigenvalues $\Gl_{k,l}$ on the ball bifurcates to form the NP eigenvalues on $\p\GO(h)$, which we denote by $\Gl_{k,l}(h)$.

The purpose of this paper is to prove that the first variation of the NP eigenvalue on the ball is in equilibrium. We prove the following theorem:
\begin{theorem}\label{Sphere perturbation}
Suppose that $\p\GO$ is a sphere. Then it holds that
\beq\label{equil}
\sum_{l=-k}^k \frac{d}{dh}{\Gl_{k,l}}(0)=0, \quad k=1,2, \ldots.
\eeq
for all $a \in C^\infty(\p\GO)$.
\end{theorem}

The study of this paper is motivated by a well-known (and challenging) question regarding the NP eigenvalues.
One can see from \eqnref{eigensphere} that if $\p\GO$ is a sphere, then
$$
\sum_{l=-k}^k \Gl_{k,l}= \frac{1}{2}.
$$
In other words, for each positive integer $k$ there are $2k+1$ eigenvalues whose sum is $1/2$.
Martensen \cite[Theorem 1]{Ma} proved that this holds to be true for the NP eigenvalues on ellipsoids. Since NP eigenvalues on ellipsoids are given by zeros of Lam\'e polynomials, there is a canonical way for grouping eigenvalues. Regarding this question, which may properly be referred to as the $1/2$-problem, the natural class of domains to be considered would be perturbations of balls. Theorem \ref{Sphere perturbation} makes us be inclined to a positive answer to the $1/2$-problem even though it does not give a definite answer.

We also discuss about the connection of Theorem \ref{Sphere perturbation} to some other conjectures regarding NP eigenvalues: the $1/6$ conjecture and the one related to the minimal Schatten norm which are regarding an extremal property of the ball.

We prove Theorem \ref{Sphere perturbation} using the Hadamard-type variational formula for the NP eigenvalues obtained by Grieser \cite{Grieser}. We also prove and use a new summation formula for spherical harmonics, which are eigenfunctions of the NP operator on a ball.

This paper is organized as follows: In the next section we recall Grieser's variational formula. Section \ref{sec: equilibrium} is to prove Theorem \ref{Sphere perturbation}.
In section \ref{conjectures and conclusion} we discuss about two conjectures related to Theorem \ref{Sphere perturbation}.

%%%%%%%%%%%%%%%%%%%%%%%%%%%%%%%%%%%%%%%%%%%%%%%%%%%%%%%%%%%%%%%%%%%%%%%%%%%%
\section{Hadamard-type variation formulas}\label{sec:Hadamard variation}
%%%%%%%%%%%%%%%%%%%%%%%%%%%%%%%%%%%%%%%%%%%%%%%%%%%%%%%%%%%%%%%%%%%%%%%%%%%%

The following Hadamard-type variational formula for NP eigenvalues was derived by Grieser \cite{Grieser}. Actually, the formula derived in that paper was for plasmonic eigenvalues. We clarify the relations between plasmon eigenvalues and NP eigenvalues, after stating the formula. We emphasize that the meaning of analytic dependency in the following theorem, especially that of eigenfunctions, needs to be carefully clarified. Here we simply refer to \cite{Grieser} for precise meaning of the analytic dependency.

\begin{theorem}[The first Hadamard variation formula \cite{Grieser}]\label{Main theorem1}
Let $\Gl \not=0, 1/2$ be an eigenvalue of the NP operator $\Kcal^*_{\p \GO}$ with eigenspace $E$. Then there are $h_0>0$ and real analytic functions $h\mapsto \Gl^{(i)}(h)$, $h\mapsto e^{(i)}(h)$ defined for $|h|<h_0$, $i=1, \ldots, {\dim}\ E$, such that for each $h$ the numbers $\Gl^{(i)}(h)$, $i=1, \ldots, {\rm dim}\ E$, are eigenvalues of $\Kcal^*_{\p\GO(h)}$ with eigenfunctions $e^{(i)}(h)$ and $\{e^{(1)}(0), \ldots, e^{{\rm dim} E}(0)\}$ is a basis of $E$.

For a fixed analytic branch $\Gl(h)$, $e(h)$ of eigenvalues and eigenfunctions with
$\Gl(0)=\Gl$ and $\Vert \nabla u(h) \Vert_{L^2(\GO(h))}=\Vert \nabla\Scal_{\p\GO(h)}[e(h)] \Vert_{L^2(\GO(h))}=1$ for each $h$,
we have
\beq\label{First variation formula}
\frac{d}{dh}\Gl(0)=\int_{\p \GO} a\Big[\Big(\Gl-\frac{1}{2} \Big)|\nabla_{\p} u|^2+\Big(\Gl+\frac{1}{2} \Big) (\p_n u|_-)^2\Big] dS.
\eeq
Here $u=\Scal_{\p\GO}[e(0)]$ and the index $-$ indicates the limits (to $\p\GO$) from $\GO$. $\nabla_{\p}$ and $\p_n$ denote the decomposition of the standard (Euclidean)  gradient $\nabla$, namely, $\nabla u=\nabla_{\p} u+ (\p_n u){\bf n}$ on $\p\GO$.
\end{theorem}

A real number $\Ge$ is called a {\it plasmonic eigenvalue} if the following problem admits a solution $u$ in the space $H^1(\Rbb^3)$:
\beq\label{plasmon}
\begin{cases}
\GD u =0 \quad\quad\quad\quad\, &\text{in}\  {\Rbb}^3\backslash \p\GO ,\\
u|_{-}=u|_{+} \quad\quad\quad\ &\text{on}\ \p\GO ,\\
\Ge\p_n u|_{-}= -\p_n u|_{+} \quad\ \, &\text{on}\ \p\GO.
\end{cases}
\eeq

Write the solution $u$ to \eqnref{plasmon} as $u(x) = \Scal_{\p \GO}[e](x)$ for some function $e \in H^{-1/2}(\p\GO)$ with $\int_{\p\GO} e dS=0$. Then the first and the second conditions in \eqnref{plasmon} are automatically fulfilled. One can see from \eqnref{singlejump} that the third condition is equivalent to
$$
\Ge \left( -\frac{1}{2} I + \Kcal_{\p \GO}^* \right)[e] = \left( \frac{1}{2} I + \Kcal_{\p \GO}^* \right)[e]
$$
that is,
$$
\Kcal_{\p \GO}^* [e] = \Gl e,
$$
where
\beq\label{relation plasmon and NP eigenvalues1}
\Ge=\frac{-\Gl-1/2}{\Gl-1/2} \quad\mbox{or}\quad \Gl=\frac{\Ge-1}{2(\Ge+1)}.
\eeq
The relation between the plasmonic eigenvalue $\Ge$ and the NP eigenvalue $\Gl$ is given by \eqnref{relation plasmon and NP eigenvalues1}.

The relation \eqnref{relation plasmon and NP eigenvalues1} shows that the plasmonic eigenvalues on the sphere is also in equilibrium. In fact, let $\Ge_{k,l}$ be the plasmonic eigenvalues on the sphere $\p\GO$ and $\Ge_{k,l}(h)$ be those on $\p\GO(h)$. Then one can see easily from \eqnref{relation plasmon and NP eigenvalues1} that
$$
\frac{d}{dh} \Ge_{k,l} (0)= \frac{1}{(\Gl_{k,l}-1/2)^2} \frac{d}{dh} \Gl_{k,l} (0).
$$
Thus we obtain the following corollary from Theorem \ref{Sphere perturbation}:
\begin{cor}
Suppose that $\p\GO$ is a sphere. Then it holds that
\beq\label{equil2}
\sum_{l=-k}^k \frac{d}{dh}{\Ge_{k,l}}(0)=0, \quad k=1,2, \ldots.
\eeq
for all $a \in C^\infty(\p\GO)$.
\end{cor}

Using the relation \eqnref{relation plasmon and NP eigenvalues1} one can easily find the NP eigenvalues on the sphere. Let $Y_{k,l}$ be a spherical harmonic of order $k$ and let
$$
u(x):=
\begin{cases}
r^k Y_{k,l}(\Go) \quad&\mbox{if } r<1, \\
r^{-k-1} Y_{k,l}(\Go) \quad&\mbox{if } r>1.
\end{cases}
$$
Then $u$ is a solution to \eqnref{plasmon} with $\Ge= \frac{k+1}{k}$. Thus we see from \eqnref{relation plasmon and NP eigenvalues1} that $\frac{1}{2(2k+1)}$ is an NP eigenvalue.

%%%%%%%%%%%%%%%%%%%%%%%%%%%%%%%%%%%%%%%%%%%%%%%%%%%%%%%%%%%%%%%%%%%%%%%%%%%%%
\section{Proof of Theorem \ref{Sphere perturbation}}\label{sec: equilibrium}
%%%%%%%%%%%%%%%%%%%%%%%%%%%%%%%%%%%%%%%%%%%%%%%%%%%%%%%%%%%%%%%%%%%%%%%%%%%%%

The eigenfunctions corresponding the NP eigenvalues on the sphere are spherical harmonics. It is convenient to use the following complex form of spherical harmonics:
\beq
u_{k, l}=r^k Y_{k, l}= C_{k, l} \sum_{\substack{p+q+s=k \\ p-q=l
\\ p, q, s\geq 0}} \frac{1}{p! q! s!}\left( -\frac{w}{2}\right)^p \left(\frac{\overline{w}}{2}\right)^q z^s,
\eeq
where
\beq
C_{k, l}=\sqrt{\frac{2k+1}{4\pi}(k+l)!(k-l)!},
\eeq
and $w=x+iy$ and $\overline{w}=x-iy$ (see \cite{VMK}). The following identity is known as Uns\"old's theorem (see \cite{WW}):
\beq\label{unsold}
\sum_{l=-k}^k |Y_{k, l}|^2=\frac{2k+1}{4\pi}.
\eeq
We prove the following proposition, which is new to the best of our knowledge.
\begin{prop}\label{formulas of spherical harmonics}
\beq
\sum_{l=-k}^k |\nabla u_{k, l}(x)|^2=\frac{k(2k+1)^2}{4\pi},  \quad x\in S^2.
\eeq
\end{prop}
\begin{proof}
We first note that
\beq
|\nabla u|^2=2 \left[ \Big|\frac{\p u}{\p w} \Big|^2 + \Big| \frac{\p u}{\p \overline{w}} \Big|^2 \right]+ \Big|\frac{\p u}{\p z} \Big|^2.
\eeq
We then have
\begin{align*}
\frac{\p u_{k, l}}{\p w}&=-\frac{C_{k, l}}{2} \sum_{\substack{p+q+s=k \geq 0 \\ p-q=l \geq 0
\\ p\geq 1 \\ q, s \geq 0}} \frac{1}{(p-1)! q! s!}\left( -\frac{w}{2}\right)^{p-1}
\left(\frac{\overline{w}}{2}\right)^q z^s \\
&=-\frac{C_{k, l}}{2 C_{k-1, l-1}} \cdot C_{k-1, l-1} \sum_{\substack{\tilde{p}+q+s=k-1 \\ \tilde{p}-q=l-1 \\ \tilde{p}, q, s \geq 0
}} \frac{1}{\tilde{p}! q! s!}\left( -\frac{w}{2}\right)^{\tilde{p}} \left(\frac{\overline{w}}{2}\right)^q z^s \\
&=-\frac{C_{k, l}}{2 C_{k-1, l-1}} r^{k-1} Y_{k-1, l-1}=-\frac{\sqrt{(k+l-1)(k+l)(2k+1)}}{2\sqrt{2k-1}} r^{k-1} Y_{k-1, l-1}.
\end{align*}
Similarly, we have
\begin{align*}
\frac{\p u_{k, l}}{\p \overline{w}}&=\frac{C_{k, l}}{2} \sum_{\substack{p+q+s=k \geq 0 \\ p-q=l \geq 0 \\ q \geq 1 \\ p, s \geq 0}} \frac{1}{p! (q-1)! s!}\left( -\frac{w}{2}\right)^{p}
\left(\frac{\overline{w}}{2}\right)^{q-1} z^s \\
&=\frac{C_{k, l}}{2 C_{k-1, l+1}} \cdot C_{k-1, l+1}
\sum_{\substack{p+\tilde{q}+s=k-1 \\ p-\tilde{q}=l+1 \\ p, \tilde{q}, s \geq 0}} \frac{1}{p! \tilde{q}! s!}\left(-\frac{w}{2}\right)^{p} \left(\frac{\overline{w}}{2}\right)^{\tilde{q}} z^s \\
&=\frac{C_{k, l}}{2 C_{k-1, l+1}} r^{k-1} Y_{k-1, l+1}=\frac{\sqrt{(k-l-1)(k-l)(2k+1)}}{2\sqrt{2k-1}} r^{k-1} Y_{k-1, l+1}.
\end{align*}
Finally we have
\begin{align*}
\frac{\p u_{k, l}}{\p z}&={C_{k, l}} \sum_{\substack{p+q+s=k \geq 0 \\ p-q=l \geq 0
\\ s\geq 1 \\ p, q \geq 0}} \frac{1}{p! q! (s-1)!}\left( -\frac{w}{2}\right)^{p}
\left(\frac{\overline{w}}{2}\right)^q z^{s-1} \\
&=\frac{C_{k, l}}{C_{k-1, l}} \cdot C_{k-1, l} \sum_{\substack{p+q+\tilde{s}=k-1 \\ p-q=l \\ p, q, \tilde{s} \geq 0
}} \frac{1}{p! q! \tilde{s}!}\left( -\frac{w}{2}\right)^{p} \left(\frac{\overline{w}}{2}\right)^q z^{\tilde s} \\
&=\frac{C_{k, l}}{C_{k-1, l}} r^{k-1} Y_{k-1, l}=\frac{\sqrt{(k+l)(k-l)(2k+1)}}{\sqrt{2k-1}} r^{k-1} Y_{k-1, l}.
\end{align*}
It then follows that
\begin{align*}
\sum_{l=-k}^k |\nabla u_{k, l}|^2 =&\sum_{l=-k}^k 2 \left[ \Big|\frac{\p u_{k, l}}{\p w} \Big|^2 + \Big| \frac{\p u_{k, l}}{\p \overline{w}} \Big|^2 \right]+ \Big|\frac{\p u_{k, l}}{\p z} \Big|^2 \\
=& \sum_{l=-k}^k 2 \Big[ \frac{(k+l-1)(k+l)(2k+1)}{4(2k-1)}|Y_{k-1, l-1}|^2 \\
& \quad\quad + \frac{(k-l-1)(k-l)(2k+1)}{4(2k-1)}|Y_{k-1, l+1}|^2 \Big] \\
&\quad + \sum_{l=-k}^k \frac{(k+l)(k-l)(2k+1)}{2k-1}|Y_{k-1, l}|^2 .
\end{align*}
Renumbering $l-1$ to $l$ in the first term above and $l+1$ to $l$ in the second, we have
\begin{align*}
\sum_{l=-k}^k |\nabla u_{k, l}|^2
=& \sum_{l=-k+1}^{k-1} \Big[ 2\left\{\frac{(k+l)(k+l+1)(2k+1)}{4(2k-1)}+
\frac{(k-l)(k-l+1)(2k+1)}{4(2k-1)} \right\} \\
&\quad\quad\quad +\frac{(k+l)(k-l)(2k+1)}{2k-1}\Big]|Y_{k-1, l}|^2 \\
=& \sum_{l=-k+1}^{k-1} \frac{k(2k+1)^2}{2k-1}|Y_{k-1, l}|^2 \\
=& \frac{k(2k+1)^2}{2k-1} \cdot \frac{2k-1}{4\pi} =\frac{k(2k+1)^2}{4\pi}.
\end{align*}
Here we used the Uns\"old's theorem for the second to the last equality.
\end{proof}

\begin{proof}[Proof of Theorem \ref{Sphere perturbation}]
One can easily see that
\beq\label{1000}
\Scal_{\p\GO}[Y_{k, l}](x) = -\frac{k}{2k+1} r^k Y_{k, l} = -\frac{k}{2k+1} u_{k,l}(x), \quad x \in \GO.
\eeq
Choose $C_k$ so that
$$
C_k^2 \int_{\GO} |\Scal_{\p\GO}[Y_{k, l}](x) |^2 dx=1.
$$
Then, according to \eqnref{First variation formula} and \eqnref{1000}, we have
\begin{align*}
\sum_{l=-k}^k \frac{d\Gl_{k, l}}{dh}(0) &= \left( \frac{C_k k}{2k+1} \right)^2 \sum_{l=-k}^k \int_{S^2}  a
\Big[\frac{-k}{2k+1}|\nabla_\p u_{k, l}|^2 + \frac{k+1}{2k+1}(\p_n u_{k, l}|_-)^2 \Big]\ dS \\
&= \left( \frac{C_k k}{2k+1} \right)^2 \sum_{l=-k}^k \int_{S^2}  a
\Big[\frac{-k}{2k+1}|\nabla u_{k, l}|^2 + (\p_n u_{k, l}|_-)^2 \Big]\ dS. 
\end{align*}
Since $\p_n u_{k,l}|_- =kY_{k, l}$ on $S^2$, it follows from \eqnref{unsold} and \eqnref{formulas of spherical harmonics} that
\begin{align*}
\sum_{l=-k}^k \frac{d\Gl_{k, l}}{dh}(0) &= \left( \frac{C_k k}{2k+1} \right)^2 \int_{S^2}  a  \sum_{l=-k}^k
\Big[-\frac{k}{(2k+1)}|\nabla u_{k, l}|^2 + k^2 |Y_{k, l}|^2 \Big]\ dS \\
&= \left( \frac{C_k k}{2k+1} \right)^2 \int_{S^2}  a \Big[-\frac{k}{(2k+1)}\cdot\frac{k(2k+1)^2}{4\pi} + k^2\cdot \frac{2k+1}{4\pi} \Big]\ dS =0,
\end{align*}
as desired.
\end{proof}

%%%%%%%%%%%%%%%%%%%%%%%%%%%%%%%%%%%%%%%%%%%%%%%%%%%%%%%%%%%%%%%%
\section{Some conjectures}\label{conjectures and conclusion}
%%%%%%%%%%%%%%%%%%%%%%%%%%%%%%%%%%%%%%%%%%%%%%%%%%%%%%%%%%%%%%%%

We now discuss about the connection of Theorem \ref{Sphere perturbation} with two conjectures raised in \cite{Miyanishi:2015aa}.

\begin{conjecture}[$1/6$-conjecture] \label{1/6 conjecture}
This conjecture claims that
\beq
{\rm max}\; \left(\sigma(\Kcal_{\p\GO})
\backslash \left\{\frac{1}{2} \right\} \right) \ge \frac{1}{6},
\eeq
for any $C^{\infty}$ simply connected closed surfaces $\p\GO$, and the equality is achieved if and only if $\p \GO$ is a sphere.
\end{conjecture}

Suppose that we attempt to prove the $1/6$-conjecture for the small perturbation of a ball. Then it is enough to show that
\beq\label{1/6}
\GL(h):= \sum_{l=-1}^1 \Gl_{1,l}(h) = \frac{1}{2},
\eeq
since then one of $\Gl_{1,-1}$, $\Gl_{1,0}$ and $\Gl_{1,1}$ is larger than (or equal to) $1/6$. Note that $\GL(0)=1/2$ and Theorem \ref{Sphere perturbation} shows that $\GL'(0)=0$. We emphasize that \eqnref{1/6} is a special case (the case when $k=1$) of the $1/2$-problem  discussed in Introduction.

The other conjecture is regarding the Schatten norm of the NP operator:
\begin{conjecture}[Minimizing Schatten norms are achieved by $S^2$]\label{Spectral zeta}
The conjecture claims that for $p>2$
\beq\label{Schatten}
{\rm tr}\{(\Kcal_{\p\GO}^* \Kcal_{\p\GO})^{p/2} \} \ge 2^{-p}(1-2^{-p})\zeta(p-1)
\eeq
for any $C^{\infty}$ simply
connected closed surface $\p\GO$ where $\zeta(s)$ denotes the Riemann zeta function, and the equality is achieved if and only if $\p \GO$ is a sphere.
\end{conjecture}

We first mention that the right-hand of \eqnref{Schatten} is the Schatten $p$-norm of the NP operator on the sphere. Suppose again that we attempt to prove \eqnref{Schatten} for $\GO(h)$, the perturbation of a ball. Recall that the NP operators in three dimensions are in Schatten class of $p>2$ (see \cite{KPS}). Thus it follows from Weyl's inequality that
$$
\sum_{\Gl_j \in \sigma(\Kcal_{\p\GO}^*)} |\Gl_j|^p\leq {\rm tr}\{(\Kcal_{\p\GO}^* \Kcal_{\p\GO})^{p/2} \} <+\infty.
$$
The spectral zeta function of $\Kcal^{*}_{\p\GO(h)}$ is given by
$$
\zeta_{\Kcal^{*}_{\p\GO(h)}}(p) := \sum_{k=0}^{\infty}\sum_{l=-k}^k \Gl_{k, l}(h)^p .
$$
We then obtain Theorem\ref{Sphere perturbation} that
\begin{align*}
\frac{d}{dh} \zeta_{\Kcal^{*}_{\p\GO(h)}}(p)\Big|_{h=0} &= \sum_{k=0}^{\infty}\sum_{l=-k}^k p \Gl_{k, l}^{p-1}(h) \frac{d}{dh}{\Gl_{k, l}(h)}\Big|_{h=0} \\
&=\sum_{k=0}^{\infty}\sum_{l=-k}^k p \cdot \left( \frac{1}{2(2k+1)} \right)^{p-1} \frac{d}{dh}{\Gl_{k, l}}(0)=0.
\end{align*}
It means that $\zeta_{\Kcal^{*}_{\p\GO(h)}}(p)$ is in equilibrium on $S^2$. We emphasize that the above computations are formally, since we have not proved differentiability of the infinite sum there.

%%%%%%%%%%%%%%%%%%%%%%%%%%%%%%%%%%%%%%%%%%%%%%%%%%%%%%%%%%%%%%%%%%%

\end{document}